\documentclass[a4paper,12pt]{article}
\usepackage[margin=1.5in]{geometry}
\usepackage{lineno}
\usepackage{amsmath,amssymb,amsthm}
\usepackage{caption}
\usepackage{subcaption}
\usepackage{ascmac,bm,framed,fancybox,mathrsfs}
\usepackage{graphicx,xcolor}
\usepackage[capitalize]{cleveref}
\usepackage{empheq}

\newtheorem{thm}{Theorem}[section] \crefname{theorem}{theorem}{theorem}
\newtheorem{prop}[thm]{Proposition} \crefname{proposition}{proposition}{proposition}
\newtheorem{lem}[thm]{Lemma} \crefname{lemma}{lemma}{lemma}
 \crefname{corollary}{corollary}{corollary}
\theoremstyle{definition}
\newtheorem{dfn}[thm]{Definition} \crefname{definition}{definition}{definition}

\theoremstyle{remark}
\newtheorem{rem}[thm]{Remark} \crefname{remark}{remark}{remark}

\numberwithin{equation}{section}

\def\Xint#1{\mathchoice
{\XXint\displaystyle\textstyle{#1}}%
{\XXint\textstyle\scriptstyle{#1}}%
{\XXint\scriptstyle\scriptscriptstyle{#1}}%
{\XXint\scriptscriptstyle\scriptscriptstyle{#1}}%
\!\int}
\def\XXint#1#2#3{{\setbox0=\hbox{$#1{#2#3}{\int}$}
\vcenter{\hbox{$#2#3$}}\kern-.5\wd0}}

\def\dashint{\Xint-}

\newcommand{\uep}[1]{u_{\epsilon}#1}

\newcommand{\Set}[2]{\left\{#1\mathrel{}\middle|\mathrel{}#2\right\}}
\newcommand{\R}{\mathbb{R}}
\newcommand{\ri}[0]{\right}
\newcommand{\lef}[0]{\left}

\begin{document}

\title{A game interpretation for the weighted $p$-Laplace equation \thanks{This work was supported by JST SPRING, Grant Number JPMJSP2119.}
}

\author{Mamoru Aihara
\thanks{Department of Mathematics, Graduate School of Science, Hokkaido University, email: \texttt{aihara.mamoru.p7@elms.hokudai.ac.jp}
}
}

            
\maketitle

\begin{abstract}
In this paper,
we obtain a stochastic approximation that converges to the viscosity solution of the weighted $p$-Laplace equation.
We consider a stochastic two-player zero-sum game controlled by a random walk, two player's choices, and the gradient of the weight function.
The proof is based on the boundary conditions in the viscosity sense and the comparison principle.
These results extend previous findings for the non-weighted $p$-Laplace equation [Manfredi, Parviainen, Rossi, 2012].
In addition,
we study the limiting behavior of the viscosity solution of the weighted $p$-Laplace equation as $p\rightarrow\infty$.
\end{abstract}

\noindent
\textbf{Keywords:}
stochastic games,
two-player zero-sum games,
viscosity solution,
weighted $p$-Laplacian

\vskip\baselineskip
\noindent
\textbf{MSC2020 Classification:} 35J92, 35J70, 35D40, 49L20, 49N70, 91A15

\section{Introduction}
In this paper, we consider a Dirichlet problem for the weighted $p$-Laplace equation:
\begin{empheq}[left={\empheqbiglbrace}]{alignat=2}
\label{A-lap}
\tag{P}
& -\Delta_{p,f} u(x)=0, & \quad  &    x \in \Omega, \\
\label{A-lapF}
\tag{F}
& u(x)=F(x), & \quad  &  x \in \partial \Omega,
\end{empheq}
where
$\Omega \subset \mathbb{R}^N$ is a bounded domain and
$p>2$ is a fixed constant.
Here $u: \overline{\Omega} \rightarrow \mathbb{R}$ is an unknown function,
$Du$ is the gradient of $u$ and 
\begin{align*}
    \Delta_{p,f} u(x):=\mathrm{div}(f(x)|Du(x)|^{p-2} Du(x)).
\end{align*}
Throughout this paper,
we assume $F$ is a continuous function defined in a neighborhood of the boundary and $f \in C^1(\Omega)$ is uniformly positive and bounded from above,
i.e.,
\begin{align*}
    0 < \inf_\Omega f \leq \sup_\Omega f < \infty .
\end{align*}

The relationship between stochastic methods and partial differential equations has yielded many intriguing results.
One of the most classical examples is the connection between the Laplace equation and Brownian motion.
Brownian motion is a stochastic process characterized by a random movement at small time intervals.
When problems such as hitting points and hitting times are considered in this context,
the Laplace equation naturally arises.
Specifically,
the distribution of hitting points and the escape time from the domain of Brownian motion correspond to the solution of the Laplace equation.
As a result,
stochastic techniques have become increasingly valuable for analyzing both the solutions and the qualitative properties of the Laplace equation.

Manfredi,
Parviainen and Rossi extended this result to the $p$-Laplace type equations in \cite{MPR12}. 
They approximated the viscosity solutions of the equation using a two-player zero-sum game called ``Tug-of-war''.
Specifically,
they proved that the value function that represents the expected value in the game uniformly converges to the viscosity solution.
Furthermore,
in recent years,
this has been extended to the normalized $p$-Laplace equation with non-homogeneous terms, parabolic type equations and obstacle problems \cite{R16, H22, LM17}.
Applications of these developments include analysis of the regularity of the solution using games and the proof of preservation of convexity \cite{LPS13,LSZ16}.
Our main result of this paper is to construct the game and prove the same result for the weighted $p$-Laplace equation.

Our game is controlled by the random walk, two player's choices, and the gradient of the weight function.
This is an extension of the game in \cite{MPR12} with a new operation.
As a result,
an additional term appears in the key formulas such as the dynami programming principle.
Although further discussions are required on the additional term,
the proof can be established by the method in \cite{DMP22, LPS13} using the comparison theorem and the half-relaxed limit.

In addition,
we consider the case as $p\rightarrow\infty$.
For the non-weighted $p$-Laplace equation,
it is known in \cite{BDM91} that its viscosity solution converges to the viscosity solution of the following infinity-Laplace equation. 
    \begin{align}
    \label{infty}
       -\Delta_\infty u(x):=-D^2u(x) Du(x)\cdot Du(x)= 0,
    \end{align}
where $D^2u$ is a Hesse matrix of $u$.
Furthermore,
in the case of the $p(x)$-Laplace equation,
it is also known in \cite{MRU10} that the solution converges to the solution of the infinity-Laplace equation with an additional term.
We obtain similar results that the viscosity solution of the weighted $p$-Laplace equation converges to the viscosity solution of the infinity-Laplace equation \eqref{infty}.

Before stating the main result,
we outline the assumptions.
\begin{itemize}
\item[$(\Omega1)$]
A bounded domain $\Omega$ satisfies the exterior cone condition,
i.e.,
for any $x\in \partial\Omega$,
there is a normal vector $n\in \R^n$ and constants $\theta \in (0,\frac{\pi}{2}]$ and $r>0$ such that 
\begin{align*}
    \{ y \in \R^n\ |\ (y-x)\cdot n \geq |y-x|\cos \theta,\ |y-x|\leq r \} \cap \overline{\Omega}=\{x\}. 
\end{align*}
\item [$(F1)$]
A boundary function $F$ is a Lipschitz continuous and the Lipschitz constant is less than or equal to one.
\end{itemize}

The main theorems are the following:
\begin{thm}
\label{main1}
Let $u_\epsilon$ be the value function given by the game and $\overline{u}$ be the upper half-relaxed limit of $\uep$ and $\underline{u}$ be the lower half-relaxed limit of $\uep$.
Then, 
$\overline{u}$ is the generalized viscosity subsolution and $\underline{u}$ is the generalized viscosity supersolution of \eqref{A-lap}\eqref{A-lapF}.
\end{thm}

\begin{thm}
\label{main2}
Assume $(\Omega 1)$.
Let $u_\epsilon$ be the value function given by the game and $u$ be the viscosity solution of \eqref{A-lap}\eqref{A-lapF}.
Then
\[ u_\epsilon \rightarrow u \quad \mathrm{uniformly}\ \mathrm{in}\ \overline{\Omega} \]
as $\epsilon \rightarrow 0$.
\end{thm}

\begin{thm}
\label{main3}
    Assume $\partial\Omega$ is $C^1$ and $(F1)$.
    Let $u_p$ be the viscosity solution of \eqref{A-lap}\eqref{A-lapF} and $u_\infty$ be the viscosity solution of \eqref{infty}\eqref{A-lapF}.
    Then 
    \[u_p \rightarrow u_\infty\quad\mathrm{uniformly}\ \mathrm{in}\ \overline{\Omega}\]
    as $p\rightarrow\infty.$
\end{thm}

\section{Preliminaries}
In this section,
we introduce the definition of the viscosity solution and several related theorems.
First,
we present the definitions of the viscosity solution and the generalized viscosity solution.
In simple terms,
the difference of these definitions lies in whether the solution satisfies the boundary condition or satisfies either the boundary condition or the equation on the boundary.
The relationship between these solutions is known as Theorem \ref{thG06}.
\begin{dfn}[Viscosity solution]
A lower (resp.\ upper) semicontinuous function $u:\Omega\rightarrow\R$ is a viscosity supersolution (resp.\ subsolution) of \eqref{A-lap}
if whenever $x_0 \in \Omega$ and $\phi \in C^2(\Omega)$ satisfy $\phi(x_0)=u(x_0)$ and
$\phi(x)>u(x)$ (resp.\ $\phi(x)<u(x)$) for $x \neq x_0$,
we have 
\begin{align*} 
- \mathrm{div} \left(f(x_0)|D\phi(x_0)|^{p-2} \phi(x_0) \right) \geq 0\quad (\mathrm{resp}.\ \leq 0).
\end{align*}
\end{dfn}
\begin{dfn}[Viscosity solution]
A lower (resp.\ upper) semicontinuous function $u:\overline{\Omega}\rightarrow\R$ is a viscosity supersolution (resp.\ subsolution) of \eqref{A-lap}\eqref{A-lapF}
if $u$ is a viscosity supersolution (resp.\ subsolution) of \eqref{A-lap} and
\begin{align*} 
u(x)-F(x) \geq 0\quad (\mathrm{resp}.\ \leq 0),\quad & x \in \partial \Omega.
\end{align*}
\end{dfn}

\begin{dfn}[Generalized viscosity solution]
A lower (resp.\ upper) semicontinuous function $u:\overline{\Omega}\rightarrow\R$
is a generalized viscosity supersolution (resp.\ subsolution) of \eqref{A-lap}\eqref{A-lapF}
if the following condition holds:
whenever $x_0 \in \overline{\Omega}$ and $\phi \in C^2(\overline{\Omega})$ satisfy $\phi(x_0)=u(x_0)$ and
$\phi(x)>u(x)$ (resp.\ $\phi(x)<u(x)$) for $x \neq x_0$,
we have if $x_0 \in \Omega$ then
\begin{align*} 
- \mathrm{div} \left(f(x_0)|D\phi(x_0)|^{p-2} \phi(x_0) \right) \geq 0\quad (\mathrm{resp}.\ \leq 0),
\end{align*}
and if $x_0 \in \partial\Omega$ then
\begin{align*}
\max \Bigl\{- \mathrm{div} \left(f(x_0)|D\phi(x_0)|^{p-2} \phi(x_0) \right) ,u(x_0)-F(x_0) \Bigl\} \geq 0\quad (\mathrm{resp}.\ \leq 0).
\end{align*}
\end{dfn}

\begin{thm}[{\cite[Theorem 2.2]{G06}}]
\label{thG06}
Assume $(\Omega 1)$.
Then the following holds:
Let $u$ be a generalized viscosity supersolution (resp.\ subsoluton) of \eqref{A-lap}\eqref{A-lapF},
then $u$ is a viscosity supersolution (resp.\ subsolution) of \eqref{A-lap}\eqref{A-lapF}.
\end{thm}

Next,
let us mention the comparison theorem under the assumptions of the main theorem.
The weighted $p$-Laplace equation is not considered in \cite{FZ22},
but a comparison theorem can be proven through a similar argument.
\begin{thm}
\label{comparison}
Let $u:\overline{\Omega} \rightarrow \R$ be a viscosity subsolution of \eqref{A-lap} and $v:\overline{\Omega} \rightarrow \R$ be a viscosity supersolution of \eqref{A-lap}.
If
\begin{align*}
 u \geq v \quad \mathrm{on}\ \partial \Omega,
\qquad \mathrm{then}\qquad
 u \geq v \quad \mathrm{in}\  \Omega.
\end{align*}
\end{thm}

\section{Rule of the game}
In this section,
we will first explain the rules of the game and define the value function determined by the game.
After that,
we will discuss some properties related to stochastic processes and then show that the value function satisfies the dynamic programming principle.

For a constant $p>2$ and the weight function $f \in C^1(\Omega)$,
we define the constants $\alpha \in (0,1)$, $\beta \in (0,1)$ as
\[ \alpha :=\frac{p-2}{p+n},\ \ \beta := 1-\alpha. \] 
Fix $\epsilon>0$.
We define the function $\gamma :\Omega \rightarrow (0,1]$ as
\[ \gamma (x):= \frac{f(x)}{\frac{|Df(x)|}{2(p+n)}+f(x)},\]
and the exterior neighborhood of $\Omega$ as
\[
\Gamma_\epsilon:=\{ y\in \R^n \setminus \Omega \ |\ \mathrm{dist}(y,\Omega) \leq \epsilon \}.
\]

We consider a two-player zero-sum game with the following rules;
\begin{enumerate}
\item First, a token is placed at $x_0 \in \Omega\cup\Gamma_\epsilon$. 
\item When the token is in position $x_k$, 
the position $x_{k+1}$ is determined by performing one of the following operations;
\begin{itemize}
\item With probability $\frac{\gamma(x_k)\alpha}{2}$,
Player A moves the token to any position in $B_\epsilon (x_k)$.
\item With probability $\frac{\gamma(x_k)\alpha}{2}$,
Player B moves the token to any position in $B_\epsilon (x_k)$.
\item With probability $\gamma(x_k)\beta$,
move the token to a uniformly random position in $B_\epsilon (x_k)$.
\item With probability $1-\gamma(x_k)$,
move the token to $x_k+\epsilon^2 \frac{Df(x_k)}{|Df(x_k)|}$.
\end{itemize}
\item Repeat the operation, moving the token $x_0$, $x_1$, $x_2$, \ldots until the token exits from the boundary for the first time.
We denote a first exiting point by $x_\tau \in \Gamma_\epsilon$. 
Then we get a payoff $F(x_\tau)$. 
\item Player A tries to maximize the payoff, and Player B tries to minimize it.
We define the expected value of the payoff as the value function, denoted as $u_\epsilon(x_0)$.
\end{enumerate}

The game described above involves probability.
Therefore, let us define the probability space following \cite{MPR12}.

First,
the history of a token that has moved $k$ times is denoted by 
a vector $(x_0,\ldots,x_k)$.
Here,
let $H_k$ be the set of all history of the token that has moved $k$ times.
Using this,
we can denote the set of all history of the token that has moved a finite number of times by
\[ H:= \bigcup_{k=0}^\infty H_k. \]
The strategy of Player A is defined as a function on $H$ as
\[ S_A(x_0,\ldots,x_k)=x_{k+1}  \in B_\epsilon(x_k) \]
and
the strategy of Player B can be defined as $S_B$ in the same way.

Next, we consider the space
$\Omega_\epsilon:=\Omega \cup \Gamma_\epsilon$ with a natural topology and
the $\sigma$-algebra of the Lebesgue measurable sets.
The product space $H^\infty:= \Omega_\epsilon \times \Omega_\epsilon \times \cdots $ has the product topology.

Let $\mathcal{B}$ be the Borel $\sigma$-algebra.
Let $\mathcal{F}_k \ (k=0,1,\ldots)$ be a $\sigma$-algebra generated by $A_0 \times A_1 \times \cdots \times A_k \times \Omega_\epsilon \times \cdots$
where
$A_i \in \mathcal{B}$
and
$\mathcal{F}_\infty$ be a $\sigma$-algebra defined by $\bigcup_{k=0}^\infty \mathcal{F}_k$.

For $\omega:=(\omega_0,\omega_1,\ldots) \in H^\infty$,
we define the function $x_k:H^\infty \rightarrow \R^n$ as
\[ x_k(\omega) = \omega_k,\ k=0,1,\ldots .\]
Then $x_k$ is $\mathcal{F}_k$-measurable random variable.
We denote a stopping time by
\[ \tau(\omega):= \inf \{ k\ |\ x_k(\omega) \in \Gamma_\epsilon,\ k=0,1, \ldots \}. \] 

Fix the initial point $x_0 \in \Omega_\epsilon$.
Let us construct the probability measure for $(H^\infty,\mathcal{F}_\infty)$.
For $A\in \mathcal{B}$,
we define the initial distribution $\delta_{x_0}(A)$ and the families of translation probabilities from $(x_0(\omega),x_1(\omega),\ldots,x_k(\omega))$ to $A$ as
\begin{align*}
 \pi_{S_A,S_B} (x_0(\omega),x_1(\omega),\ldots&,x_k(\omega),A) =  \pi_{S_A,S_B}(\omega_0,\omega_1\ldots,\omega_k,A) \\
& = \gamma(\omega_k)\beta \frac{|A\cap B_\epsilon(\omega_k)|}{|B_\epsilon(\omega_k)|} 
+\gamma(\omega_k) \frac{\alpha}{2} \delta_{S_A(\omega_0,\ldots,\omega_k)} (A) \\
& \quad +\gamma(\omega_k) \frac{\alpha}{2} \delta_{S_B(\omega_0,\ldots,\omega_k)} (A)  +\Delta(A),
\end{align*}
where
\begin{align*}
\Delta(A):=
\begin{cases}
(1-\gamma(\omega_k)) \delta_{\omega_k+\epsilon^2 \frac{Df(\omega_k)}{|Df(\omega_k)|} }(A), \quad & Df(\omega_k)\neq0, \\
0, & Df(\omega_k)=0.
\end{cases}
\end{align*}

Using these,
we can obtain the probability measure on the finite product as follows.
\begin{align*}
& \mu_{S_A,S_B}^{0,x_0}(A_0) = \delta_{x_0}(A_0),   \\ 
& \mu_{S_A,S_B}^{k,x_0}(A_0 \times \cdots  \times A_k) \\
& \quad=\int_{A_0 \times \cdots \times A_{k-1}} \pi_{S_A,S_B}(\omega_0,\omega_1\ldots,\omega_{k-1},A_k) d\mu_{S_A,S_B}^{k-1,x_0}(\omega_0,\ldots,\omega_{k-1}),
\end{align*}
where $k=1,2,\ldots$.
By Kolmogorov's extension theorem,
we can define the probability measure on $(H^\infty,\mathcal{F}_\infty)$.

For the initial point $x_0$,
we denote the value of Player A by
\[u_{\epsilon}(x_0):=u_A^\epsilon (x_0) :=\sup_{S_A} \inf_{S_B}  \mathbb{E}_{S_A,S_B}^{x_0} [ F(x_\tau)]  ,    \]
and the value of Player B by
\[u_B^\epsilon (x_0):= \inf_{S_B} \sup_{S_A} \mathbb{E}_{S_A,S_B}^{x_0} [ F(x_\tau)], \]
where $\mathbb{E}_{S_A,S_B}^{x_0} [ F(x_\tau)]$ is the expected value when Player A chooses the strategy $S_A$ and Player B chooses the strategy $S_B$.

\begin{rem}
As a consequence of Theorem \ref{th.DPP2},
we have
$u_A^\epsilon (x_0)=u_B^\epsilon (x_0)$.
\end{rem}

Now, let us introduce the key equation for the proof.
The following equation for a measurable function $w$ is called the {\it dynamic programming principle}. 
Here,
$\dashint$ denotes the mean integral.
\begin{align*}
\label{DPP}
\tag{DPP}
\begin{split}
w(x)= \gamma (x)\frac{\alpha}{2} \Bigl( \sup_{B_{\epsilon}(x) } w + \inf_{B_{\epsilon}(x) } w \Bigl) + \gamma (x) \beta  \dashint_{B_{\epsilon}(0)} w(x+h) dh  
 +U(w) (x),
\end{split}
\end{align*}
where,
$U$ is defined as
\begin{align*}
U(w)(x):=
\begin{cases}
(1-\gamma(x)) w\Bigl(x+\epsilon^2 \frac{Df(x)}{|Df(x)|} \Bigl), \quad& Df(x) \neq 0, \\
0, & Df(x)=0.
\end{cases}
\end{align*}

We deal with the upper and lower relaxed limits as a limit of the value function.
\begin{align*}
\label{relax}
& \overline{u}(x):=\lim_{\nu \rightarrow 0} \sup \Set{ \uep(y) }{  y\in\overline{\Omega}, |y-x|+\epsilon \leq \nu }, \\
& \underline{u}(x):=\lim_{\nu \rightarrow 0} \inf \Set{ \uep(y) }{  y\in\overline{\Omega}, |y-x|+\epsilon \leq \nu } .
\end{align*}

Now, let us introduce the concept of the supermartingale and the optional stopping theorem in the context of stochastic processes. 
Roughly speaking,
supermartigale is the property in which the expected value of the random variable decreases as the operation is repeated. 
\begin{dfn}[Supermartingale]
Let $(X,\mathcal{F},P)$ be a probability space.
A stochastic processes $\{ M_k(w)\}_{k=1}^{\infty},w \in X$ is called {\it supermartingale} with respect to $\{ \mathcal{F}_k\}_{k=1}^\infty\subset \mathcal{F},\ \mathcal{F}_1 \subset \mathcal{F}_2 \subset \cdots $ if the following conditions hold:
\begin{enumerate}
\item Each $M_k$ is measurable with respect to $\mathcal{F}_k$.
\item $\mathbb{E}[|M_k|] < \infty$.
\item $\mathbb{E}[M_k|\mathcal{F}_{k-1} ] \leq M_{k-1}$.
\end{enumerate}
\end{dfn}

\begin{prop}[Optimal stopping theorem]
Let $\{ M_k(w)\}_{k=1}^{\infty}$ be a supermartingale,
and let $\tau \in \mathbb{N}$ be a bounded stopping time.
Then the following holds:.
\[  \mathbb{E}[|M_\tau|] \leq \mathbb{E}[|M_0|].\]
\end{prop}

Let us prove that the value function satisfies \eqref{DPP}.
First, we construct a function that satisfies \eqref{DPP}.
\begin{lem}
\label{th.DPP1}
Let $F:\mathbb{R}^n \rightarrow \mathbb{R}$ be a bounded measurable function.
Then,
there exists a measurable function $v$ that satisfies the following conditions.
\begin{enumerate}
\item $v$ satisfies \eqref{DPP} in $\Omega$.
\item $v(x)=F(x), \quad x \in \Gamma_{\epsilon} $.
\end{enumerate}
\end{lem}
The proof of this theorem follows \cite[Theorem 2.1]{LPS14}.

\begin{proof}
(2) is clear,
so we prove (1).

step 1.
For a measurable function $w:\Omega \cup \Gamma_\epsilon \rightarrow \R$,
we define the operator $T$ by the following:
\begin{align*} Tw(x):=
\begin{cases}
 \gamma (x)\frac{\alpha}{2} \Bigl( \sup_{B_{\epsilon}(x) } w + \inf_{B_{\epsilon}(x) } w \Bigl) \\
\qquad+ \gamma (x) \beta  \dashint_{B_{\epsilon}(0)} w(x+h) dh  
+U(w)(x),  \quad &x \in \Omega, \\
F(x), \quad \ &x \in \Gamma_\epsilon.
\end{cases}
\end{align*}
Obviously,
$Tw$ is measurable.
Next,
we define
\begin{align*}
v_0(x):=
\begin{aligned}
\begin{cases}
\inf_{y \in \Gamma_\epsilon} F(y), \quad  &x \in \Omega, \\
F(x), &x \in \Gamma_\epsilon,
\end{cases}
\end{aligned}
\end{align*}
and $v_j \ (j=1,2,\ldots),\ v$ as
\[ v_{j}:=Tv_{j-1},\qquad v:=\lim_{j\rightarrow\infty}v_j. \]
From the definition,
it follows that $v_0 \leq v_1$.
Additionally,
we get
\begin{align*}
v_{j+1}(x)-v_j(x)=& \gamma (x)\frac{\alpha}{2} \Bigl( \sup_{B_{\epsilon}(x) } v_j - \sup_{B_{\epsilon}(x) } v_{j-1} \Bigl)+\gamma (x)\frac{\alpha}{2} \Bigl( \inf_{B_{\epsilon}(x) } v_j - \inf_{B_{\epsilon}(x) } v_{j-1} \Bigl) \\
& + \gamma (x) \beta  \dashint_{B_{\epsilon}(0)} (v_j-v_{j-1})(x+h) dh  \\
&  +U(v_j)(x)-U(v_{j-1})(x).
\end{align*}
From this equation and the monotonicity of $U$,
we see that $v_0 \leq v_1 \leq v_2 \leq \cdots$.
Since $F$ is bounded,
the construction of $v_j$ shows that $v_j\ (j=1,2,\cdots)$ is uniformly bounded with respect to $j$.
Thus, 
$v_j$ converges pointwise to a measurable function $v$ on $\Omega \cup\Gamma_\epsilon$ as $j\rightarrow \infty$.

step 2.
We now show that the sequence $v_j$ converges uniformly to $v$ on $\Omega \cup\Gamma_\epsilon$ as $j \rightarrow \infty$.

We proceed by contradiction. Since $v_j=F$ on $\Gamma_\epsilon$,
we suppose
\[ 
M:=\lim_{j \rightarrow \infty} \sup_{x\in \Omega}(v-v_j)(x) >0.
\]
Fix an arbitrary $\delta >0$.
Take a sufficiently large $k\in \mathbb{N}$ such that
\begin{align}
\label{DPPpf1}
v-v_k \leq M+\delta \quad \mathrm{in} \ \Omega.
\end{align}
Also,
there exists $x_0\in \Omega$ such that
\[ v(x_0)-v_{k+1}(x_0) \geq M-\delta. \] 
Since $v_j$ converges pointwise to $v$, 
we4 can take $l>k$ sufficiently large so that
 \[v(x_0)-v_{l+1}(x_0) < \delta. \]
Combining these,
we get
\begin{equation}
\label{DPPpf2}
 v_{l+1}(x_0)-v_{k+1}(x_0) \geq M-2\delta. 
\end{equation}
Since $v_j$ is uniformly bounded,
by the dominated convergence theorem,
for a sufficiently large $k>0$,
we have
\begin{equation}
\label{DPPpf3}
\sup_{x\in \Omega} \gamma(x) \beta \dashint_{B_\epsilon (x)} (v-v_{k+1})(y) dy \leq \delta.
\end{equation}

(i)
When $Df(x_0)\neq0$,
from \eqref{DPPpf1}--\eqref{DPPpf3},
we obtain
\begin{align*}
&M-2\delta  \leq  v_{l+1}(x_0)-v_{k+1}(x_0) \\
& = \gamma(x_0) \frac{\alpha}{2} \Bigl( \sup_{B_\epsilon (x_0)} v_l - \sup_{B_\epsilon (x_0)} v_k \Bigl) 
+\gamma(x_0) \frac{\alpha}{2} \Bigl( \inf_{B_\epsilon (x_0)} v_l - \inf_{B_\epsilon (x_0)} v_k \Bigl) \\
& \quad +\gamma(x_0) \beta \dashint_{B_\epsilon (x_0)} (v_l-v_k)(y) dy +(1-\gamma(x_0))(v_l-v_k)\lef(x_0+\epsilon^2 \frac{Df(x_0)}{|Df(x_0)|} \ri) \\
& \leq \gamma(x_0) \alpha  \sup_{\Omega} (v_l -  v_k )\\
&\qquad+\gamma(x_0) \beta \dashint_{B_\epsilon (x_0)} (v_l-v_k)(y) dy
+(1-\gamma(x_0))\sup_{\Omega} (v_l -  v_k )\\
& \leq \gamma(x_0) \alpha  \sup_{\Omega} (v -  v_k )\\
&\qquad+\gamma(x_0) \beta \dashint_{B_\epsilon (x_0)} (v-v_k)(y) dy
+(1-\gamma(x_0))\sup_{\Omega} (v -  v_k )\\
& \leq (\gamma(x_0)\alpha+1-\gamma(x_0))(M+\delta)+\delta.
\end{align*}
Since $\gamma(x_0)\alpha+1-\gamma(x_0) <1$,
this leads to a contradiction as $\delta \rightarrow 0$.

(ii)
When $Df(x_0)=0$,
similarly to (i),
we obtain
\begin{align*}
M-2\delta & \leq  v_{l+1}(x_0)-v_{k+1}(x_0) \\
& = \gamma(x_0) \frac{\alpha}{2} \Bigl( \sup_{B_\epsilon (x_0)} v_l - \sup_{B_\epsilon (x_0)} v_k \Bigl) 
+\gamma(x_0) \frac{\alpha}{2} \Bigl( \inf_{B_\epsilon (x_0)} v_l - \inf_{B_\epsilon (x_0)} v_k \Bigl) \\
& \quad +\gamma(x_0) \beta \dashint_{B_\epsilon (x_0)} (v_l-v_k)(y) dy  \\
& \leq \gamma(x_0) \alpha  \sup_{\Omega} (v_l -  v_k )+\gamma(x_0) \beta \dashint_{B_\epsilon (x_0)} (v_l-v_k)(y) dy \\
& \leq \gamma(x_0) \alpha  \sup_{\Omega} (v -  v_k )+\gamma(x_0) \beta \dashint_{B_\epsilon (x_0)} (v-v_k)(y) dy\\
& \leq \gamma(x_0)\alpha(M+\delta)+\delta.
\end{align*}
Since $\gamma(x_0)\alpha<1$,
this leads to contradiction as $\delta \rightarrow 0$.

Thus,
we conclude that $\lim_{j \rightarrow \infty} \sup_{x\in \Omega}(v-v_j)(x) =0$ and
$v_j$ converges uniformly to $v$.
Therefore,
by taking the limit $j \rightarrow \infty$ in the equation $ v_{j}:=Tv_{j-1}$,
we see that $v$ satisfies \eqref{DPP}.
\end{proof}

Next,
we confirm that the value functions satisfy the dynamic programming principle.
The proof follows \cite[Theorem 3.2]{LPS14}.
\begin{lem}
\label{th.DPP2}
Let $u_A^\epsilon$ be the value function of Player A,
$u_B^\epsilon$ be the value function of Player B, and
$v$ be the function that satisfies \eqref{DPP}.
Then the following holds:
\[ v= u_A^\epsilon=u_B^\epsilon .\]
\end{lem}

\begin{proof}
If we can show that $ u_B^\epsilon \leq v $,
then by a similar argument,
we can conclude that $  v \leq u_A^\epsilon$.
Furthermore,
by definition,
we have $ u_A^\epsilon \leq u_B^\epsilon $.
Thus,
we will demonstrate that 
$u_B^\epsilon \leq v$ below.

Let $x_0 \in \Omega$ and fix an arbitrary constant $\eta >0$.
Denote an arbitrary strategy of Player A by $S_A$,
and let the strategy of Player B be $S_B^0$ which moves the token to $x_k \in B_\epsilon (x_{k-1})$  satisfying the following condition whenever the token is at
$x_{k-1} \in \Omega$:
\[ v(x_k) \leq \inf_{B_\epsilon (x_{k-1})} v + \eta 2^{-k} .\]
From the definitions of the strategies $S_A$ and $S_B^0$,
and since $v$ satisfies \eqref{DPP},
we can calculate:
\begin{align*}
\mathbb{E}_{S_A,S_B^0}^{x_0} & [ v(x_k)+\eta 2^{-k}\ | x_0,\ldots,x_{k-1} ] \\
& = \gamma (x_{k-1}) \frac{\alpha}{2} \lef\{ v\left( S_A(x_0,\ldots,x_{k-1} )\right)+v\left( S_B^0(x_0,\ldots,x_{k-1} )\right) \ri\}  \\
& \quad + \gamma (x_{k-1})\beta \dashint_{B_\epsilon (x_{k-1})} v(y) dy  
+ U(v)(x_{k-1})  +\eta 2^{-k} \\
& \leq \gamma (x_{k-1}) \frac{\alpha}{2} \lef( \inf_{B_\epsilon (x_{k-1})} v + \eta 2^{-k} + \sup_{B_\epsilon (x_{k-1})} v \ri)  \\
& \quad + \gamma (x_{k-1})\beta \dashint_{B_\epsilon (x_{k-1})} v(y) dy  
+U(v)(x_{k-1})  +\eta 2^{-k} \\
& = v(x_{k-1}) +\eta 2^{-k} \lef(1+\gamma(x_{k-1})\frac{\alpha}{2} \ri) \\
& \leq v(x_{k-1}) +\eta 2^{-(k-1)}.
\end{align*}
Thus,
if we set $M_k:=v(x_k)+\eta 2^{-k}$,
$M_k$ becomes a supermartingale with respect to the strategies $S_A$ and $S_B^0$.
Let $\tau$ be the stopping time.
From the optimal stopping theorem and the fact that $F(x_\tau)=v(x_\tau)$,
we get:
\begin{align*}
u_B^\epsilon (x_0) 
 & = \inf_{S_B} \sup_{S_A} \mathbb{E}_{S_A,S_B}^{x_0} [ F(x_\tau)]  \\
& \leq \sup_{S_A} \mathbb{E}_{S_A,S_B^0}^{x_0} [ F(x_\tau)+\eta 2^{-\tau}]  \\
& \leq \sup_{S_A} \mathbb{E}_{S_A,S_B^0}^{x_0} [ F(x_0)+\eta]  \\
& =v(x_0)+\eta.
\end{align*}
Since $\eta >0$ was arbitrary,
we conclude that $u_B^\epsilon \leq v$.
\end{proof}

\section{Proofs of Theorems \ref{main1} and \ref{main2}}

\begin{proof}[Proof of Theorem \ref{main1}]
We now present only the proof for the supersolution.
Fix $x \in \Omega$,
and take $\phi \in C^2(\overline{\Omega})$.

step 1:
Choose $x_1^{\epsilon}\in B_{\epsilon}(x)$ such that it satisfies the following:
\[\phi(x_1^{\epsilon}) = \min_{\overline{B}_{\epsilon}(x)} \phi . \]
By considering the Taylor approximation of $\phi$ around $x$,
we obtain
\[ \phi(x_1^{\epsilon})=\phi(x)+ D\phi(x) \cdot (x_1^{\epsilon}-x) 
+\frac{1}{2} D^2\phi(x) (x_1^{\epsilon}-x)\cdot(x_1^{\epsilon}-x)  +o(\epsilon^2) \]
where
$\frac{o(\epsilon^2)}{\epsilon^2} \rightarrow 0\ (\epsilon \rightarrow 0)$.
Next,
take $\tilde{x}_1^{\epsilon}$ such that $\tilde{x}_1^{\epsilon}-x=-(x_1^{\epsilon}-x)$.
Using the Taylor approximation similarly,
we obtain
\[ \phi(x_1^{\epsilon})=\phi(x)- D\phi(x) \cdot (x_1^{\epsilon}-x)  
+\frac{1}{2} D^2\phi(x) (x_1^{\epsilon}-x)\cdot(x_1^{\epsilon}-x)  +o(\epsilon^2). \]
Adding these two equations together gives
\[ \phi(x_1^{\epsilon})+\phi(\tilde{x}_1^{\epsilon})-2\phi(x)=  D^2\phi(x) (x_1^{\epsilon}-x)\cdot(x_1^{\epsilon}-x)  +o(\epsilon^2). \]
Thus,
we have
\begin{equation}
\label{pf1}
\begin{split}
\frac{1}{2} \lef[ \max_{\overline{B}_{\epsilon}(x)} \phi +  \min_{\overline{B}_{\epsilon}(x)} \phi \ri] -  \phi (x)  & \geq  \frac{1}{2} \lef[ \phi(x_1^{\epsilon})+\phi(\tilde{x}_1^{\epsilon}) \ri]-\phi(x) \\ & =\frac{1}{2}  D^2 \phi(x) (x_1^{\epsilon}-x)\cdot(x_1^{\epsilon}-x)  +o(\epsilon^2). 
\end{split}
\end{equation}
Now,
by Taylor approximation of $\phi$ for $h \in \mathbb{R}^n$,
and taking the average over $h$,
we obtain
\begin{align*}
 \dashint_{B_{\epsilon}(0)} \phi &(x+h) dh\\
 &=\phi (x) + \dashint_{B_{\epsilon}(0)}  D\phi(x) \cdot h\  dh +\frac{1}{2}\dashint_{B_{\epsilon}(0)}  D^2\phi(x) \ h \cdot h\ dh+o(\epsilon^2) .
\end{align*}
Here,
due to the symmetry of $h$,
we have 
\[ \dashint_{B_{\epsilon}(0)}  D\phi \cdot h \ dh=0.\]
Using $\dashint_{B_{\epsilon}(0)}  {h_i}^2  dh =\frac{\epsilon^2}{n+2}$,
the third term can be calculated as:
\begin{align*}
\dashint_{B_{\epsilon}(0)}  D^2\phi(x) \ h \cdot h \ dh & = \dashint_{B_{\epsilon}(0)} \sum_{i,j=1}^n \phi_{x_ix_j}(x) h_i h_j  dh \\
& =\sum_{i,j=1}^n \phi_{x_ix_j}(x) \dashint_{B_{\epsilon}(0)}  h_i h_j  dh \\
& =\sum_{i=1}^n \phi_{x_ix_i}(x) \dashint_{B_{\epsilon}(0)}  {h_i}^2  dh \\
& = \Delta \phi(x) \frac{\epsilon^2}{n+2}.
\end{align*}
From this,
we obtain
\begin{equation}
\label{pf2}
 \dashint_{B_{\epsilon}(0)} \phi (x+h) dh -\phi (x) = \frac{\epsilon^2}{2(n+2)}\Delta \phi(x) +o(\epsilon^2) .
\end{equation}
By the Taylor approximation,
we obtain
\begin{equation}
\label{pf3}
\begin{split}
\phi \Bigl(x+\epsilon^2 \frac{Df(x)}{|Df(x)|}\Bigl) &=\phi (x) +  D\phi(x) \cdot \epsilon^2 \frac{Df(x)}{|Df(x)|} +o(\epsilon^2) \\
& = \phi (x) +\frac{\epsilon^2}{|Df(x)|}  D\phi(x) \cdot Df(x)  +o(\epsilon^2) .
\end{split}
\end{equation}
By calculating these expressions \eqref{pf1}--\eqref{pf3},
we get
\begin{equation}
\label{pf4}
\begin{split}
\gamma (x)\frac{\alpha}{2} & \Bigl( \max_{\overline{B}_{\epsilon}(x) } \phi + \min_{\overline{B}_{\epsilon}(x) } \phi \Bigl) 
 + \gamma (x) \beta  \dashint_{B_{\epsilon}(0)} \phi(x+h) dh  \\
 & \qquad\qquad\qquad\qquad +(1-\gamma(x)) \phi \Bigl(x+\epsilon^2 \frac{Df(x)}{|Df(x)|} \Bigl)  \\
& \geq \phi(x)+ \frac{\epsilon^2}{|Df(x)|+2f(x)(p+n)}  \Bigl\{ f(x) \Delta \phi(x)\\
& \qquad +f(x)(p-2)  D^2\phi(x) \frac{x_1^{\epsilon}-x}{\epsilon}\cdot \frac{x_1^{\epsilon}-x}{\epsilon} + D\phi(x)\cdot Df(x)  \Bigl\}
+o(\epsilon^2) .
\end{split}
\end{equation}

step 2:
Assume that $\underline{u}-\phi$ takes a locally minimum 0 at $x_0 \in \overline{\Omega}$.
Then,
we can choose a sequence $x_\epsilon \in \overline{\Omega}$ such that
\begin{align*}
&x_\epsilon \rightarrow x_0 \ (\epsilon \rightarrow 0),\\
u_\epsilon (x_\epsilon)-\phi(x_\epsilon) &\leq u_\epsilon(y)-\phi(y)+\epsilon^3, \quad  y\in B_r(x_\epsilon) .
\end{align*}
In addition,
by the definition of $\underline{u}$,
we get
\[ \uep(x_\epsilon)-\phi(x_\epsilon)=o(\epsilon^3). \]

Assume $x_\epsilon \in \Omega$.

(i)
If $Df(x_0) \neq 0$,
by the continuity of $Df$,
we have $Df(x_\epsilon) \neq0$ in a neighborhood of $x_0$.
Combining the \eqref{DPP} and \eqref{pf4},
we obtain
\begin{align*}
0&=-u_{\epsilon}(x_\epsilon)+ \gamma (x_\epsilon)\frac{\alpha}{2} \Bigl( \sup_{B_{\epsilon}(x_\epsilon) } u_{\epsilon} + \inf_{B_{\epsilon}(x_\epsilon) } u_{\epsilon} \Bigl) \\
 & \qquad\qquad\qquad\qquad + \gamma (x_\epsilon) \beta  \dashint_{B_{\epsilon}(0)} u_{\epsilon}(x_\epsilon+h) dh  
+U(\uep)(x_\epsilon)  \\
& \geq -\phi(x_\epsilon) + \gamma (x_\epsilon)\frac{\alpha}{2} \Bigl( \sup_{B_{\epsilon}(x_\epsilon) } \phi + \inf_{B_{\epsilon}(x_\epsilon) } \phi \Bigl) 
 + \gamma (x_\epsilon) \beta  \dashint_{B_{\epsilon}(0)} \phi(x_\epsilon+h) dh  \\
& \qquad\qquad\qquad\qquad+  (1-\gamma(x_\epsilon)) \phi \Bigl(x_\epsilon+\epsilon^2 \frac{Df(x_\epsilon)}{|Df(x_\epsilon)|} \Bigl)-\epsilon^3 \\
& \geq \frac{\epsilon^2}{|Df(x_\epsilon)|+2f(x_\epsilon)(p+n)} \Big\{ f(x_\epsilon) \Delta \phi \\
&    \ +f(x_\epsilon)(p-2)\frac{1}{2} D^2\phi(x_\epsilon) \frac{x_1^{\epsilon}-x_\epsilon}{\epsilon}\cdot \frac{x_1^{\epsilon}-x_\epsilon}{\epsilon} + D\phi(x_\epsilon) \cdot Df(x_\epsilon)  \Big\} +\mathcal{O}(\epsilon^3).
\end{align*}
By noting that
$\frac{x_1^{\epsilon}-x}{\epsilon} \rightarrow -D\phi (x)$ as $(\epsilon \rightarrow 0)$, 
and dividing both sides by $\epsilon^2$,
then letting $\epsilon \rightarrow 0$,
we obtain
\[ 0 \leq -\Delta_{p,f} \phi (x_0). \]

(ii)
If $Df(x_0)=0$,
combining \eqref{DPP} and \eqref{pf4},
we obtain
\begin{align*}
0&=-u_{\epsilon}(x_\epsilon)+ \gamma (x_\epsilon)\frac{\alpha}{2} \Bigl( \sup_{B_{\epsilon}(x_\epsilon) } u_{\epsilon} + \inf_{B_{\epsilon}(x_\epsilon) } u_{\epsilon} \Bigl) \\
& \qquad\qquad\qquad\qquad  +\gamma (x_\epsilon) \beta  \dashint_{B_{\epsilon}(0)} u_{\epsilon}(x_\epsilon+h) dh  
+U(\uep)(x_\epsilon)  \\
& \geq -\phi(x_\epsilon) + \gamma (x_\epsilon)\frac{\alpha}{2} \Bigl( \sup_{B_{\epsilon}(x_\epsilon) } \phi + \inf_{B_{\epsilon}(x_\epsilon) } \phi \Bigl) \\
&\qquad\qquad\qquad\qquad + \gamma (x_\epsilon) \beta  \dashint_{B_{\epsilon}(0)} \phi(x_\epsilon+h) dh  
 + U(\phi)(x_\epsilon)-\epsilon^3 \\
& \geq \frac{\epsilon^2}{|Df(x_\epsilon)|+2f(x_\epsilon)(p+n)} \Bigl\{ f(x_\epsilon) \Delta \phi (x_\epsilon) \\
& \qquad \qquad+f(x_\epsilon)(p-2)\frac{1}{2}  D^2\phi(x_\epsilon) \frac{x_1^{\epsilon}-x_\epsilon}{\epsilon}\cdot \frac{x_1^{\epsilon}-x_\epsilon}{\epsilon} +\tilde{U}(\phi)(x) \Bigl\}
+\mathcal{O}(\epsilon^3),
\end{align*}
where
\begin{align*}
\tilde{U}(\phi)(x):=
\begin{cases}  D\phi(x) \cdot Df(x) , \quad & Df(x)\neq0, \\
0, & Df(x)=0.
\end{cases}
\end{align*}
Since
\[ \lim_{\epsilon \rightarrow 0} \lef( D\phi(x_\epsilon)\cdot Df(x_\epsilon)\ri)=D\phi(x_0)\cdot Df(x_0)=0,\] 
we get
\[ \lim_{\epsilon \rightarrow 0} \tilde{U}(\phi)(x_\epsilon)=0=D\phi(x_0)\cdot Df(x_0). \]
By noting that
$\frac{x_1^{\epsilon}-x}{\epsilon} \rightarrow -D\phi (x) \ \ (\epsilon \rightarrow 0)$,
and dividing both sides by $\epsilon^2$,
then letting $\epsilon \rightarrow 0$,
we obtain
\[ 0 \leq -\Delta_{p,f} \phi (x_0). \]

Assume $x_\epsilon \in \partial \Omega$.
Then
\[ \uep(x_\epsilon)=F(x_\epsilon). \] 

Combining these,
we conclude that
\begin{align*} 
-\Delta_{p,f} \phi (x_0) \geq 0,\qquad & x_0 \in \Omega , \\
\max \Bigl\{ -\Delta_{p,f} \phi (x_0) ,u(x_0)-F(x_0) \Bigl\} \geq 0,\qquad & x_0 \in \partial \Omega.
\end{align*}
\end{proof}

\begin{proof}[Proof of Theorem \ref{main2}]
From the definition of the relaxed limit,
we have $ \underline{u} \leq \overline{u}$.
Moreover,
by Theorem \ref{main1},
we find that
$\underline{u}$ is a generalized viscosity supersolution of \eqref{A-lap}\eqref{A-lapF} and
$\overline{u}$ is a generalized viscosity subsolution of \eqref{A-lap}\eqref{A-lapF}.
By Theorem \ref{thG06},
we get
\[ \underline{u} = \overline{u}=F \quad \mathrm{on} \  \partial \Omega. \]
By the comparison principle,
we have
\[ \underline{u} \geq \overline{u} \quad \mathrm{in} \  \Omega. \]
Thus,
$\underline{u} = \overline{u}$ on $\Omega$.
From the fundamental properties of the relaxed limit in \cite[Remark 6.4]{CIL92},
it follows that $\uep$ uniformly converges to the viscosity solution of the boundary value problem \eqref{A-lap}\eqref{A-lapF}.
\end{proof}

\section{Passing to the limit $p\rightarrow\infty $} 
Let us consider taking $p \to \infty$  in rule of the game introduced in section 3. 
Then,
by the definition of $\alpha$ and $\beta$,
$\alpha \rightarrow 1$ and $\beta \rightarrow 0$.
In this case,
our game corresponds precisely to the game for the infinity-Laplace equation in \cite{PSSW09}.
Thus,
it can be conjectured that
the solution of the weighted $p$-Laplace equation converges to the solution of the infinity-Laplace equation as $ p \to \infty$.
In this section,
we prove that this convergence holds in the sense of viscosity solutions.
The proof follows the method in \cite{MRU10}.
First,
we derive the equicontinuity of $u_p$ with respect to $p$ following the lemma.

\begin{lem}
\label{equi}
    Let $\partial\Omega$ be $C^1$.
    Assume $(F1)$.
    Let $u_p$ be the viscosity solution of \eqref{A-lap}\eqref{A-lapF},
    then $\{u_p\}_p$ is equicontinuous with respect to $p$
\end{lem}
\begin{proof}
    On the basis of \cite{FZ22},
    it can be concluded that the viscosity solution and the distribution solution are equivalent.
    Furthermore,
    it is known in \cite{HKM93} that the distribution solution of the weighted $p$-Laplace equation is the minimizer of the following functional $H$ in $W^{1,p}(\Omega)$.
    \[
    H(w)=\int_{\Omega} f(x)|Dw(x)|^{p} dx.
    \]
    Here,
    fix $\alpha>n$ and assume that $p$ is large enough such that $p>\alpha$.
    Using a similar argument as for \cite{MRU10}, 
    we have $u_p \in W^{1,\alpha}(\Omega)$ and $||u_p||_{W^{1,\alpha}(\Omega)} $ is uniformly bounded with respect to $p$.
    Due to Morrey's inequality,
   we obtain
   \[
   ||u_p||_{C^{0,\gamma}(\Omega)} \leq C||u_p||_{W^{1,\alpha}(\Omega)},
   \]
   where $\gamma=1-\frac{n}{\alpha}$ and $C>0$ is the constant depending only on $\alpha,\ n $ and $\Omega$.
   Thus,
   we obtain the equicontinuity of $u_p$ with respect to $p$.
\end{proof}
Moreover,
by applying the maximum principle or considering the rule of the game,
we obtain the uniform boundedness of $u_p$.
Together with the results of Lemma \ref{equi},
it follows from Ascoli's theorem that $u_p$ has a uniformly convergent limit.
We now proceed to the proof of Theorem \ref{main3} using this uniform convergence.
\begin{proof}[Proof of Theorem \ref{main3}]
    We present only the proof for the supersolution.
    Assume that $u_\infty-\phi$ takes a local minimum at $x \in \Omega$.
    The case $D\phi (x)=0$ is obvious,
    so we assume $D\phi(x) \neq0$.
    By the uniform convergence,
    we can choose a sequence $x_p \in \Omega$ such that
    \[ x_p \rightarrow x \ (p \rightarrow \infty), \]
    \[ u_p (x_p)-\phi(x_p) \leq u_p(y)-\phi(y), \quad  y\in B_r(x_p) .\]
    By the definition of viscosity supersolution,
    \begin{align*}
    &-\Delta_{p,f}\phi(x_p) \\
    &\quad = -|D\phi(x_p)|^{p-2} \Bigl\{ 
        f(x_p) \Delta \phi(x_p) +D\phi(x_p) \cdot Df(x_p)    \\
        &\qquad\qquad+ f(x_p)(p-2)|D\phi(x_p)|^{-2} D^2\phi(x_p) D\phi(x_p) \cdot D\phi(x_p)  \Bigl\} \geq 0.
\end{align*}    Dividing both sides by $f(x_p)(p-2)|D\phi(x_p)|^{p-2}>0$ and taking the limit as $p\rightarrow \infty$,
    we obtain
    \begin{align*}
        - D^2\phi(x) D\phi(x)\cdot D\phi(x)\geq 0. 
    \end{align*}
    This provides the conclusion.
\end{proof}


\begin{thebibliography}{00}
  
  \bibitem{BDM91} 
  T. Bhattacharya, E. DiBenedetto, J. Manfredi, 
  \textit{ Limits as $p\rightarrow\infty$ of $\Delta_pu_p=f$ and related extremal problems.} 
  (Turin, 1989),
  Rend. Sem. Mat. Univ. Politec. Torino 1989, 
  Special Issue, 15–68, 
  1991.
  
\bibitem{CIL92} 
M. G. Crandall, H. Ishii, P. L. Lions,
\textit{User's guide to viscosity solutions of second order partial differential equations}, Bull. Amer. Math. Soc. (N.S.) 27,
no. 1, 1–67,
1992. 

\bibitem{DMP22}  
F. del Teso, J. J.  Manfredi, M. Parviainen, 
\textit{  Convergence of dynamic programming principles for the p-Laplacian.} 
Adv. Calc. Var. 15,
no. 2, 191–212,
2022. 

\bibitem{FZ22}
Y. Fang, C. Zhang, 
\textit{Equivalence between distributional and viscosity solutions for the double-phase equation.}
Adv. Calc. Var. 15,
no. 4, 811–829, 
2022.

\bibitem{G06} 
G. Gripenberg, 
\textit{Generalized viscosity solutions of elliptic PDEs and boundary conditions.},
Electron. J. Differential Equations, 
No. 43, 10 pp, 
2006.

\bibitem{H22}  
J. Han, 
\textit{ Time-dependent tug-of-war games and normalized parabolic p-Laplace equations.} 
Nonlinear Anal. 214,
Paper No. 112542, 23 pp
2022.

\bibitem{HKM93}
J. Heinonen, T. Kilpeläinen and O. Martio,  
\textit{Nonlinear potential theory of degenerate elliptic equations},
Oxford Mathematical Monographs. Oxford Science Publications. The Clarendon Press, Oxford University Press, New York,
vi+363 pp,
1993.

\bibitem{LM17}
M. Lewicka and J. J. Manfredi,
\textit{The obstacle problem for the $p$-laplacian via optimal stopping of tug-of-war games},
Probab. 
Theory Related Fields 167, no. 1-2, 349–378,
2017.


\bibitem{LSZ16}
Q. Liu, A. Schikorra and X. Zhou,
\textit{A game-theoretic proof of convexity-preserving properties for motion by curvature},
Indiana Univ. Math.
J. 65, no. 1, 171–197,
2016.

\bibitem{LPS13}
H. Luiro, M. Parviainen and E. Saksman, 
\textit{Harnack's inequality for $p$-harmonic functions via stochastic games},
Comm. Partial Differential Equations 38, no. 11, 1985–2003,
2013.

\bibitem{LPS14}
H. Luiro, M. Parviainen and E. Saksman,
\textit{On the existence and uniqueness of $p$-harmonious functions},
Differential Integral Equations. 27, no. 3-4, 201–216,
2014.

\bibitem{MPR12}
J. J. Manfredi, M. Parviainen and J. D. Rossi,
\textit{On the definition and properties of $p$-harmonious functions},
Ann. Sc. Norm. Super. Pisa Cl. Sci. (5) 11, no. 2, 215–241,
2012.

\bibitem{MRU10}
J. J. Manfredi, J. D. Rossi, J. M. Urbano,
\textit{Limits as $p(x)\rightarrow\infty$ of $p(x)$-harmonic functions.}
Nonlinear Anal. 72, no. 1, 309–315,
2010.

\bibitem{PSSW09}
Y. Peres, O. Schramm, S. Sheffield and D. B. Wilson, 
\textit{Tug-of-war and the infinity Laplacian},
J. Amer. Math. Soc. 22, no. 1, 167–210,
2009.


\bibitem{R16}  
 E. Ruosteenoja,
 \textit{ Local regularity results for value functions of tug-of-war with noise and running payoff.}
 Adv. Calc. Var. 9, no. 1, 1–17,
 2016.
 
\end{thebibliography}
\end{document}